\documentclass[a4paper,reqno,11pt]{amsart}
\usepackage[a4paper,margin=1in]{geometry}
\usepackage{amsmath,amssymb,amsthm}

\theoremstyle{plain}
\newtheorem{theorem}{Theorem}[section]
\newtheorem{lemma}[theorem]{Lemma}
\newtheorem{proposition}[theorem]{Proposition}

\theoremstyle{definition}
\newtheorem{remark}[theorem]{Remark}
\numberwithin{equation}{section}

\title[Solitary bifurcations]{On the first bifurcation of solitary  waves}

\author{Vladimir Kozlov}
\address{Department of Mathematics, Link\"oping University, SE-581 83 Link\"oping, Sweden}

\begin{document}
	
\begin{abstract} We consider solitary water waves on the vorticity flow in a two-dimensional channel of finite depth. The main object of study is a branch of solitary waves starting from a laminar flow and then approaching an extreme wave. We prove that there always exists a bifurcation point on such branches. Moreover, the crossing number of the first bifurcation point is $1$, i.e. the bifurcation occurs at a simple eigenvalue.


\end{abstract}

\maketitle

\section{Introduction}

We consider steady water waves in a two-dimensional channel bounded below by a flat,
rigid bottom and above by a free surface that does not touch the bottom. The surface tension is neglected and the water motion can be rotational.
In appropriate Cartesian coordinates $(X, Y )$, the bottom coincides with the
$X$--axis and gravity acts in the negative $Y$--direction. We choose the frame of reference so that the velocity field is time-independent as well as the free-surface profile
which is supposed to be the graph of $Y = \xi(X)$, $x \in \Bbb R$, where $\xi$ is a positive and
continuous unknown function. Thus
$$
D=D_\xi = \{X\in \Bbb R, 0 <Y < \xi(X)\},\;\; S=S_\xi=\{X\in\Bbb R,\;Y=\xi(X)\}
$$
is the water domain and the free surface respectively. We will use the stream function $\Psi$, which is connected with the velocity vector $({\bf u},{\bf v})$ as $\Psi_Y=c-{\bf u}$ and ${\bf v}=\Psi_X$, where $c$ is a constant wave speed.

We assume that $\xi$ is a positive,   even  function.  Since the surface tension is neglected, $\Psi$ and
$\xi$ after a certain scaling satisfy the following free-boundary problem (see, for example, \cite{KN14}):
\begin{eqnarray}\label{K2a}
&&\Delta \Psi+\omega(\Psi)=0\;\;\mbox{in $D_\xi$},\nonumber\\
&&\frac{1}{2}|\nabla\Psi|^2+\xi=R\;\;\mbox{on $S_\xi$},\nonumber\\
&&\Psi=1\;\;\mbox{on $S_\xi$},\nonumber\\
&&\Psi=0\;\;\mbox{for $Y=0$},
\end{eqnarray}
where $\omega\in C^{1,\alpha}$, $\alpha\in (0,1)$, is a vorticity function and $R$ is the Bernoulli constant. We assume that $\Psi$ is even in $X$ and
\begin{equation}\label{J27ba}
\Psi_Y>0\;\;\mbox{on $\overline{D_\xi}$},
\end{equation}
which means that the flow is unidirectional.

The main our subject of study is the solitary waves. The solitary wave solution satisfies
\begin{equation}\label{Ju31a}
\;\;\xi(X)\to d,\;\;\Psi(X,Y)\to U(Y)\;\;\mbox{as $X\to\pm\infty$}
\end{equation}
uniformly in $Y$. Here $(U,d)$ is a uniform stream solution with the Bernoulli constant  $R=d+U'(d)^2/2$. Moreover the Froude number
\begin{equation}\label{Au22b}
F=\Big(\int_0^d\frac{dY}{U'(Y)^2}\Big)^{-2}
\end{equation}
for this uniform stream solution  must be $>1$.  Every non-trivial solitary wave satisfies $\xi(X)>d$ and it is symmetric, after a certain translation, and monotonic decreasing for positive $X$. In what follows we assume that such translation is already done and we consider even solitary waves.

If the Froude number $F$ is equal to $1$ or equivalently $R=R_c$, then the only solution of (\ref{K2a}) is the uniform stream solution $(U_*(Y),d_*)$, where $d_*=d(s_c)$ (see Sect. \ref{SJ6} for details). 

The Frechet derivative for the problem is evaluated for example in \cite{KL2}, \cite{Koz1} in the periodic case. The same evaluation can be used to obtain the Frechet derivative in the case, when the solution is a solitary wave.  The corresponding eigenvalue problem for the  Frechet derivative has the form
\begin{eqnarray}\label{J17ax}
&&\Delta w+\omega'(\Psi)w+\mu w=0 \;\;\mbox{in $D_\xi$},\nonumber\\
&&\partial_n w-\rho w=0\;\;\mbox{on $S_\xi$},\nonumber\\
&&w=0\;\;\mbox{for $Y=0$},
\end{eqnarray}
where $n$ is the unite outward normal to $Y=\xi(X)$ and
\begin{equation}\label{Sept17aa}
\rho=
\rho(X)=\frac{(1+\Psi_X\Psi_{XY}+\Psi_Y\Psi_{YY})}{\Psi_Y(\Psi_X^2+\Psi_Y^2)^{1/2}}\Big|_{Y=\xi(X)}.
\end{equation}
The function $w$ in (\ref{J17ax}) is supposed also to be even. 

Let us introduce several function spaces. For $\alpha\in (0,1)$ and $k=0,1,\ldots$, the space $C^{k,\alpha}(D)$ consists of bounded  functions in $D$ such that the norms $C^{k,\alpha}(D_{a,a+1})$ are uniformly bounded with respect to $a\in\Bbb R$. Here
$$
D_{a,a+1}=\{(X,Y)\in \overline{D},\;:\,a\leq x\leq a+1\}.
$$
The space    $C^{k,\alpha}_{0}(D)$  consists of function in $C^{k,\alpha}(D)$  vanishing for $Y=0$. The space $C^{k,\alpha}_{0, e}(D)$ consists of even functions from $C^{k,\alpha}_0(D)$.
Similarly we define the space    $C^{k,\alpha}_{ e}(\Bbb R)$) consisting of even functions in $C^{k,\alpha}(\Bbb R)$.

We will consider a branch of solitary water waves depending on a parameter $t> 0$, i.e.
\begin{equation}\label{J3b}
\Psi=\Psi(X,Y;t),\;\;\xi=\xi(X,t),\;\;R=R(t),\;\;t\in (0,\infty),
\end{equation}
where $R$ is the Bernoulli constant.

We assume that the above sequence of solitary waves satisfies the conditions of the following theorem, which follows from   \cite{W}, if we take into account Theorem 6.1 in \cite{W2}.

\begin{theorem}\label{T1}  Fix a H{\"o}lder exponent $\alpha\in(0, 1/2]$. There exists a continuous curve {\rm (\ref{J3b})}
of solitary wave solutions to {\rm (\ref{K2a})} with the regularity
\begin{equation}\label{Au30a}
(\Psi(t),\xi(t), R(t))) \in C^{2,\alpha}( D(t))\times C^{2,\alpha}({\Bbb R})\times (R_c,\infty),
\end{equation}
where $D(t)$ denotes the fluid domain corresponding to $\xi(t)$. The solution curve satisfies  the following property (critical laminar flow, i.e. $R=R_c$)
\begin{equation}\label{Ju18aa}
\lim_{t\to 0}(\Psi(t),\xi(t),R(t))=(U_*,d_*,R_c).
\end{equation}

Furthermore, there exists a sequence $\{t_j\}_{j=1}^\infty$ such that $t_j\to\infty$ as $j\to\infty$ and $(X_j,Y_j)\in D(t_j)$ such that the following  is valid
\begin{equation}\label{Au25a}
 (\Psi(t_j))_Y(X_j,Y_j)\to 0\;\; \mbox{as $j\to\infty$}\;\;\;\mbox{(Stagnation)}.
 \end{equation}


\end{theorem}

\begin{remark} {\rm (i)} The convergence in (\ref{Ju18aa}) is understood in the spaces
$$
C^{2,\alpha}(D(t)) \times C^{2,\alpha}({\Bbb R})\times \Bbb R
$$

{\rm (ii)} The dependence on $t\in (0,\infty)$ in (\ref{Au30a}) is analytic in the variables $(q,p)$ using in partial hodograph transformation, see \cite{W2}, Sect.6.

{\rm (iii)} There are two more options in the above theorem, see Theorem 1.1 in \cite{W}. The first one says that the Froude number (equivalently the Bernoulli constant) can be arbitrary large along the curve (\ref{J3b}). This option is excluded by Proposition \ref{Pr24a}. The second one includes existence of non-trivial solitary waves with the Froude number $1$ or equivalently with the Bernoulli constant $R_c$. By Theorem 1 in \cite{KLN1} the only solution of (\ref{K2a}) is the uniform stream solution. This fact excludes the second option.

{\rm (iv)} In paper \cite{W} the Froude number $F(t)$ is used instead of $R(t)$. Since the dependence $R$ on $F$ is monotone and analytic (see Proposition \ref{Pr26b}) we use here the function $R(t)$.
\end{remark}

In Sect. \ref{S23a} we prove the following clarification of the stagnation condition (\ref{Au25a}):

\begin{proposition}\label{Pr26} The stagnation condition {\rm (\ref{Au25a})} is equivalent to the following options

{\rm (i)} If $R_0=\infty$ then  there exists a sequence $\{t_j\}_{j\geq 1}$, $t_j\to\infty$ as $j\to\infty$  such that
\begin{equation}\label{Au29ba}
\lim_{j\to\infty}|\xi'(t_j)(0)|=\infty
\end{equation}
or
\begin{equation}\label{Au29bb}
\lim_{j\to\infty}R(t_j)=R\;\;\mbox{and}\;\;\lim_{j\to\infty}\xi(t_j)(0)=R.
\end{equation}

{\rm (ii)} If $R_0<\infty$ then besides the options in {\rm (i)} there is one more option:
\begin{equation}\label{Au29bc}
\lim_{j\to\infty}R(t_j)=R\;\;\mbox{and}\;\;\lim_{j\to\infty}\sup_{X\in\Bbb R}\Psi_Y(t_j)(X_j,Y_j)=0,
\end{equation}
where $Y_j\to 0$ as $j\to \infty$.

\end{proposition}

Denote by ${\mathcal A}(t)={\mathcal A}_{(\Psi,\xi)}$ the unbounded operator in $L^2(D_\xi)$ corresponding to the problem (\ref{J17ax}), i.e.
${\mathcal A}_{(\Psi,\xi)}=-\Delta -\omega'(\Psi)$ with the domain
$$
{\mathcal D}(t)={\mathcal D}({\mathcal A}_{(\Psi,\xi)})=\{w\in H^2(D_\xi)\,:\,w(x,0)=0,\;\;\partial_\nu w-\rho w=0\, \mbox{for}\; Y=\xi(X)\}.
$$

For every $t>0$ we denote by
\begin{equation}\label{Au27b}
(U(t),d(t),R(t))
\end{equation}
 the limit uniform stream solution and Bernoulli constant in the asymptotics (\ref{Ju31a}).

Consider  the eigenvalue problem
\begin{eqnarray*}
&& -v^{''}-\omega'(U(t))v=\nu v\;\;\mbox{on $(0,d(t))$},\nonumber\\
&&v(0)=0\;\;\mbox{and}\;\;v'(d(t))-\rho_0(t)v(d(t))=0,
\end{eqnarray*}
where
\begin{equation}\label{Okt25a}
\rho_0=\frac{1+U_YU_{YY}}{U_Y^2}\Big|_{Y=d}.
\end{equation}
As it is shown in Lemma \ref{LJu20a}, the eigenvalues of this problem are positive for all $t>0$. We denote by $\nu_0=\nu_0(t)$ the lowest eigenvalue.  This eigenvalue is simple and corresponding eigenfunctions does not change sign on $(0,d(t)]$.

The following proposition is possibly known but we present its proof in Sect. \ref{S31a}  for readers convenience
\begin{proposition}\label{PrOkt24a} The operator (Frechet derivative)  ${\mathcal A}$ has a continuous spectrum $[\nu_0,\infty)$. The spectrum on the half line $(-\infty,a]$ consists of isolated eigenvalues of finite multiplicity for any $a<\nu_0$.
\end{proposition}

The main result of this paper is the following

\begin{theorem}\label{T1a}
{\rm (i)} The lowest eigenvalue of the operator ${\mathcal A}(t)$ for $t>0$  is always negative and simple with the eigenfunction which does not change sign in $D(t)$.
We denote it by $\mu_0(t)$.

{\rm (ii)} For small positive $t$ the spectrum of the operator ${\mathcal A}(t)$ on the half line $\mu\leq 0$ consists of the  eigenvalue $\mu_0(t)$ only.

{\rm (iii)} If for a certain $t_1>0$ and $t\in (0,t_1)$ the operator ${\mathcal A}(t)$  has the only negative eigenvalue $\mu_0$ and ${\mathcal A}(t_1)$ has the eigenvalue $\mu=0$ then this eigenvalue is simple.

\end{theorem}

Denote by $\mu_1(t)$ the second eigenvalue of the operator ${\mathcal A}(t)$ on the half-line $(-\infty,\nu_0)$ if it exists, otherwise we put $\mu_1(t)=\nu_0(t)$.
Assume that
\begin{eqnarray}\label{Okt30a}
&&\mbox{the existence of point $t_*>0$ such that: (i) $\mu_1(t)>0$ for $t\in (0,t_*)$},\nonumber\\
&&\mbox{ $\mu_1(t)<0$ for small positive $t-t_*$}.
\end{eqnarray}
According to Theorem \ref{T1a}(iii) the eigenvalue $\mu_1(t_*)=0$ is simple. This implies that the crossing number of the Frechet derivative is $1$ at $t_*$.
\begin{proposition}\label{PrN0v20} Let the condition {\rm (\ref{Okt30a})} hold. Then $t_*$ is a bifurcation point and bifurcating solution builds a connected continuum.
\end{proposition}
The proof of this assertion follows from Theorem I.16.4,\cite{Ki1}, (see also \cite{Ki1a}). First one must use the partial hodograph transformation to obtain a problem with an analytic operator function and then apply result of \cite{Ki1} on the local bifurcation under the condition (\ref{Okt30a}). Now one can extend this local bifurcation to a global bifurcation branch by using bifurcation results proved  in \cite{W} and  \cite{W2}.


In the next proposition we give a condition when (\ref{Okt30a}) is fulfilled

 \begin{proposition}\label{PrNov11}
 Let there exists a sequence $\{t_j\}_{j=1}^\infty$ such that the curve {\rm (\ref{J3b})} approach an extreme solitary wave with the angle\footnote{Actually the condition on the angle can be removed, see Remark \ref{PrN26}. Since the paper \cite{KL5} is not published yet we will keep this angle requirement.} $120^\circ$  at the crest along this sequence then there exists a point $t_*$ which satisfies the condition {\rm (\ref{Okt30a})}.
 \end{proposition}
Indeed, by using \cite{KL2} together with Theorem 3.1 in \cite{Koz1} one can show that the number of negative eigenvalues of the operator ${\mathcal A}(t_j)$ tends to infinity when $j\to\infty$.
This implies (\ref{Okt30a}) and hence Proposition \ref{PrN0v20}.




The first study of bifurcations on the branches of solitary water waves was done by Plotnikov in \cite{P3} in the irrotational case. In the case of vortical flows one cannot apply the same technique as in irrotational case since it is based on the complex analysis approach. The important contributions allowing to approach the  bifurcation problems are papers  \cite{W} and \cite{W2}, where the global branches of solutary waves were constructed, and the papers \cite{KL2} together with \cite{Koz1} and \cite{Koz1a}, where the bifurcation analysis for branches of Stokes waves were presented.

\section{Uniform stream solution, dispersion equation, the Bernoulli and Froude constants}\label{SJ6}

\subsection{Uniform stream solution and dispersion equation}\label{SJ6a}
The uniform stream solution $\Psi=U(Y)$ with the constant depth $\xi =d$  satisfies the problem
\begin{eqnarray}\label{X1}
&&U^{''}+\omega(U)=0\;\;\mbox{on $(0,d)$},\nonumber\\
&&U(0)=0,\;\;U(d)=1,\nonumber\\
&&\frac{1}{2}U'(d)^2+d=R.
\end{eqnarray}
To find solutions to this problem we introduce a parameter $s=U'(0)$. We assume that
 $s>s_0:=\sqrt{2\max_{\tau\in [0,1]}\Omega(\tau)}$, where
$$
\Omega(\tau)=\int_0^\tau \omega(p)dp.
$$
Then the problem (\ref{X1}) has a solution $(U,d)$ with a strongly monotone function $U$ for
\begin{equation}\label{M6ax}
R={\mathcal R}(s):=\frac{1}{2}s^2+d(s)-\Omega(1).
\end{equation}
The solution is given by
\begin{equation}\label{F22a}
Y=\int_0^U\frac{d\tau}{\sqrt{s^2-2\Omega(\tau)}},\;\;d=d(s)=\int_0^1\frac{d\tau}{\sqrt{s^2-2\Omega(\tau)}}.
\end{equation}
The function $d$ is strongly decreasing. Since $U'(Y)^2=s^2-2\Omega(U)(Y)$ and $0\leq U(Y)\leq 1$, we have
\begin{equation}\label{Okt19a}
s^2-s_0^2\leq U'(Y)^2\leq s^2-2\min_{\tau\in [0,1]}\Omega(\tau).
\end{equation}

If we consider (\ref{M6ax}) as the equation with respect to $s$ then it is solvable if $R\geq R_c$, where
\begin{equation}\label{F27a}
R_c=\min_{s\geq s_0}{\mathcal R}(s),
\end{equation}
and it has two solutions if
$R\in (R_c,R_0)$, where
\begin{equation}\label{D19ba}
R_0={\mathcal R}(s_0).
\end{equation}
We denote by $s_c$ the point where the minimum in (\ref{F27a}) is attained.

\begin{proposition} (\cite{KLN1}, \cite{KN11b},  \cite{L1})

{\rm (a)} If $R<R_c$ then the problem {\rm (\ref{K2a})} has no solutions;

{\rm (b)} If $R=R_c$  then the only solution to {\rm (\ref{K2a})} is $(U_*(Y;d_*),d_*)$;

{\rm (c)} If $R_0$ is finite and $R\geq R_0$ then there are no non-trivial solitary wave solutions to {\rm (\ref{K2a})};

{\rm (d)} If $R\in (R_c,R_0) $ then a solitary wave solution to {\rm (\ref{K2a})} satisfies
$$
\xi(0)>d_+.
$$

\end{proposition}

One more formula for the Froude number is the following
\begin{equation}\label{Au31a}
\frac{1}{F^2(s)}=-\frac{d'(s)}{s},
\end{equation}
where the Froude number $F(s)$ corresponds to the uniform stream  solution $(U(Y;s),d(s))$ and $R={\mathcal R}(s)$. One can verified directly by using
(\ref{F22a}), that
$$
\Big(\frac{d'(s)}{s}\Big)'>0.
$$
Therefore the function $F(s)$ is increasing and $F(s)\to\infty$ as $s\to\infty$. Furthermore
\begin{equation}\label{Ju28a}
{\mathcal R}'(s)=s(1-F^{-2}(s))
\end{equation}
The above formula implies that the function ${\mathcal R}$ is a strongly increasing function of  $s$, $s>s_c$, and ${\mathcal R}(s)\to\infty$ as $s\to\infty$. Therefore every parameter  $s>s_c$, $R>R_c$ or $F>1$  can be expressed through another  and can be used  for parametrisation of the solitary waves. In this case
$$
s\in (s_c,\infty),\;\;R\in (R_c,\infty)\;\;\mbox{and}\;\;F\in (1,\infty).
$$
The equation ${\mathcal R}(s)=R$ with $R\leq R_c$ has exactly one solution on $[s_c,\infty)$. We denote it by $s_-$. We put
$$
d_-=d(s_-),\;\;F_-=F(s_-).
$$
\begin{proposition}\label{Pr26b}
The Bernoulli constant $R\in [R_c,\infty)$ can be considered as a continuous function of the Froude number $F\in [1,\infty)$. This function is strongly increasing, analytical on $(1,\infty)$ and $R(1)=R_c$. Moreover
$$
R(F)=\frac{1}{2}F^{4/3}+O(1)\;\;\mbox{as $F\to\infty$}.
$$

\end{proposition}
\begin{proof} The function $d(s)$ is strongly decreasing and analytic for $s>s_0$. Furthermore $d(s)=s^{-1}(1+O(s^{-2})$ for large $s$. Therefore, ${\mathcal R}(s)$ is strongly increasing and analytic for $s>s_c$. For large $s$, ${\mathcal R}(s)=s^2/2+O(1)$. From (\ref{Au31a}9 it follows that $F(s)$ ia stronglu increasing and analytic for $s>s_c$. For large $s$, $F(s)=s^{3/2}(1+O(s^{-2})$. The above considerations lead to the proof of our proposition.

\end{proof}

Existence of small amplitude Stokes waves is determined by the dispersion equation (see, for example, \cite{KN14}). It is defined as follows.
 The strong monotonicity of $U$ guarantees that the problem
\begin{equation}\label{Okt6b}
\gamma^{''}+\omega'(U)\gamma-\tau^2\gamma=0,\;\; \gamma(0,\tau)=0,\;\;\gamma(d,\tau)=1
\end{equation}
has a unique solution $\gamma=\gamma(y,\tau)$ for each $\tau\in\Bbb R$, which is even with respect to $\tau$ and depends analytically on $\tau$.
Introduce the function
\begin{equation}\label{Okt6ba}
\sigma(\tau)=\kappa\gamma'(d,\tau)-\kappa^{-1}+\omega(1),\;\;\kappa=U'(d).
\end{equation}
It depends also analytically on $\tau$ and it is strongly increasing with respect to $\tau>0$. Moreover, it is an even function.
The dispersion  equation  is the following
\begin{equation}\label{Okt6bb}
\sigma(\tau)=0.
\end{equation}
It has a positive solution if
\begin{equation}\label{D17a}
\sigma(0)<0.
\end{equation}
By \cite{KN14} this is equivalent to $s+d'(s)<0$ or what is the same
\begin{equation}\label{D19b}
1<\int_0^d\frac{dY}{U'^2(Y)}.
\end{equation}
 By (\ref{Au22b}) the inequality (\ref{D19b}) means that $F<1$, which is well-known condition for existence of  Stokes waves of small amplitude.
Another equivalent formulation is given by requirement (see, for example \cite{KN11})
\begin{equation}\label{M3aa}
s\in (s_0,s_c).
\end{equation}
The existence of such $s$ is guaranteed by $R\in (R_c,R_0)$. If (\ref{D17a}) holds then equation (\ref{Okt6bb}) has a unique positive root, which will be denoted by $\tau_*$. It is connected with the corresponding period $\Lambda_0$ by the relation
$$
\tau_*=\frac{2\pi}{\Lambda_0}.
$$

The value $\sigma(0)$ admits the following representation (see \cite{KN14}):
$$
\sigma(0)=-\frac{3}{2\kappa}\frac{{\mathcal R}'(s)}{d'(s)}=\frac{3(F^2(s)-1)}{2\kappa}.
$$
If $F>1$ then
\begin{equation}\label{Ju10b}
\int_0^d(|v'|^2-\omega'(U)|v|^2)dY-\rho_0|v(d)|^2>0\;\;\mbox{for all $v\in H^1(0,d)$\,:\, $v(0)=0$}.
\end{equation}


To give another representation of the function $\sigma$ we introduce
\begin{equation}\label{F28a}
\rho_0=\frac{1+U'(d)U^{''}(d)}{U'(d)^2}
\end{equation}
and note that
$$
\frac{1+U'(d)U^{''}(d)}{U'(d)^2}=\kappa^{-2}-\frac{\omega(1)}{\kappa}.
$$
Hence another form for (\ref{Okt6ba}) is
\begin{equation}\label{M21aa}
\sigma(\tau)=\kappa\gamma'(d,\tau)-\kappa\rho_0.
\end{equation}

The following problem will be used in asymptotic analysis of a branch of Stokes waves  of small amplitude (see Sect.  \ref{SNov11a} and \ref{SJu5}):
\begin{eqnarray}\label{J7b}
&&v^{''}+\omega'(U)v-\tau^2v=f\;\;\mbox{on $(0,d)$},\nonumber\\
&&v'(d)-\rho_0v(d)=g\;\;\mbox{and}\;\;v(0)=0.
\end{eqnarray}
\begin{proposition} Let $\tau\geq 0$ and $\tau\neq\tau_*$. Let also $f\in C^{1,\alpha}([0,d])$ and $g$ be a constant. Then the problem
{\rm (\ref{J7b})}  has a unique solution $v\in C^{3,\alpha}$.
If $\tau=\tau_*$ then the problem {\rm (\ref{J7b})} has the one dimensional kernel which consists of function
$$
c\gamma(Y;\tau_*).
$$
\end{proposition}

\subsection{Estimates for the Bernoulli constant and the Froude number}\label{SectA29}

The aim of this section is to prove the following.

\begin{proposition}\label{Pr24a} {\rm (i)} If $R_0<\infty$ then for existence of non-trivial solitary wave solutions to {\rm (\ref{K2a})} the inequality $R<R_0$ must be valid.

{\rm (ii)} Let $R_0=\infty$.
There exist constants $\widehat{R}$ and $\widehat{F}$ depending on $\omega$ such that if $R>\widehat{R}$ (or $F>\widehat{F}$), then the problem  {\rm (\ref{K2a})} has no solitary wave solutions satisfying {\rm (\ref{J27ba})}.
\end{proposition}
\begin{proof}
We start from estimating the Bernoulli constant.

The following classification of vorticity functions was introduced in \cite{KN11}. According to that classification
the following
three options describing vorticity behaviour are possible:

(i) Either $\max_{\tau\in [0,1]}\Omega(\tau)$ is attained at an inner point of $(0,1)$ or the maximum is attained at one
of the end-points (or at both), and in the latter case, one (or both) of the following conditions
holds:
$$
\omega(1) = 0\mbox{\;\; when $\Omega(1)>\Omega(\tau)$ for $\tau\in (0,1)$};
$$
$$
\omega(0)=0\;\; \mbox{when $\Omega(0)>\Omega(\tau)$ for $\tau\in (0,1)$}.
$$
(ii) The function $\Omega$ attains its maximum only at the interval’s left end-point in which case
$$
\Omega(0) >\Omega(\tau)\;\;\mbox{ for $\tau\in (0,1]$ and $\omega(0) < 0$}.
$$
(iii) The function $\Omega$ attains its maximum when $\tau = 1$, that is, $\Omega(\tau)<\Omega(1)$ for $\tau\in (0,1)$, and
$\omega(1)> 0$.

\bigskip
In the cases (ii) and (iii) $d(s_0)<\infty$ and the number $R_0=\mathcal{R}(s_0)$ is also finite. It was proved in \cite{KLN1} that if $R\geq R_0$ then the problem (\ref{K2a}) has no nontrivial solutions. This implies the estimate $R< R_0$ for existence of non-trivial solutions to (\ref{K2a}) and in particular solitary waves. This proves the assertion (i) of Proposition \ref{Pr24a}.

In the case (i), $d(s_0)=\infty$ and the equation $\mathcal{R}(s)=R$ has two roots $s_+<s_c<s_-$ for all $R>R_c$.
The Froude number depends on the uniform stream solution. If we fix the vorticity and the Bernoulli constant $R$ then there are two solutions to the problem (\ref{X1}). We denote then by $(U_-,d_-)$ and $(U_+,d_+)$. They correspond to $s=s_+$ and $s=s_-$ respectively. Corresponding Froude numbers we denote by $F_\pm$, then
\begin{equation}\label{Okt25ba}
\frac{1}{F_\pm^2}=\int_0^{d_\pm}\frac{dY}{U_\pm'^2(Y)}.
\end{equation}
Hence
$$
F_+<1<F_-.
$$
The depths
$$
d_-=d(s_-),\;\;d_+=d(s_+)
$$
are called supercritical and subcritical respectively.

Let us turn to a nontrivial solitary solution $(\Psi,\xi)$. By \cite{KLN1,KN11b}
\begin{equation}\label{Au9a}
\xi(0)>d_+.
\end{equation}
Consider the flow forth invariant
\begin{equation}\label{Au6a}
S=\int_0^{\xi(X)}\Big(\Psi_Y^2-\Psi_X^2+R-Y+\int_\Psi^1\omega(\tau)d\tau\Big)dY.
\end{equation}
This expression does not depend on $X$. Evaluating it for $X=0$ and for $X\to\infty$ we get
\begin{equation}\label{Au6aa}
\int_0^{d+a}\Big(\Psi_Y^2+R-Y+\int_\Psi^1\omega(\tau)d\tau\Big)dY=\int_0^{d}\Big(U_y^2+R-Y+\int_U^1\omega(\tau)d\tau\Big)dY,
\end{equation}
where the function $\Psi$ is evaluated at the point $X=0$, $d+a=\xi(0)$ and $d=d_-$. We write (\ref{Au6aa}) as
\begin{equation}\label{Au6b}
\int_0^{d+a}\Psi_Y^2dY+Ra-\frac{a(2d+a)}{2}=\int_0^{d}\Big(U_y^2+\int_U^1\omega(\tau)d\tau\Big)dY-\int_0^{d+a}\int_\Psi^1\omega(\tau)d\tau dY.
\end{equation}
Using inequalities $d+a\leq R$ and
$$
\Big(\int_0^{d+a}\Psi_YdY\Big)^2\leq (d+a)\int_0^{d+a}\psi_Y^2dY,
$$
we get
$$
\frac{1}{d+a}\leq \int_0^{d+a}\Psi_Y^2dY.
$$
Therefore (\ref{Au6b}) implies
\begin{equation}\label{Au6ba}
\frac{1}{d+a}+\frac{aR}{2}-\frac{ad}{2}\leq\int_0^{d}\Big(U_Y^2+\int_U^1\omega(\tau)d\tau\Big)dY-\int_0^{d+a}\int_\Psi^1\omega(\tau)d\tau dY.
\end{equation}
Next
\begin{eqnarray}\label{Au6bb}
&&\int_0^{d}\int_U^1\omega(\tau)d\tau dY-\int_0^{d+a}\int_\psi^1\omega(\tau)d\tau dY=
\int_0^{d}\int_U^1\omega(\tau)d\tau dY-\int_0^{d}\int_\psi^1\omega(\tau)d\tau dY\nonumber\\
&&+\int_0^{d}\int_\psi^1\omega(\tau)d\tau dY-\int_0^{d+a}\int_\psi^1\omega(\tau)d\tau dY\leq a\omega_0+\int_0^d\int_\psi^U\omega(\tau)d\tau dY\nonumber\\
&&\leq a\omega_0+d\omega_0.
\end{eqnarray}
Therefore
\begin{equation}\label{Au13a}
\frac{1}{d+a}+\frac{aR}{2}-\frac{ad}{2}\leq\int_0^{d}U_Y^2dY+ a\omega_0++d\omega_0
\end{equation}
or
\begin{equation}\label{Au13aa}
\frac{aR}{2}\leq \frac{ad}{2}+\int_0^{d}\Big(U_Y^2-\frac{1}{d^2}\Big)dY+\frac{1}{d}-\frac{1}{d+a}+ (a+d)\omega_0,
\end{equation}
which implies
\begin{equation}\label{Au13b}
\frac{R}{2}\leq \frac{d}{2}+\frac{1}{a}\int_0^{d}\Big(U_Y^2-\frac{1}{d^2}\Big)dY+\frac{1}{d(d+a)}+ \frac{(a+d)\omega_0}{a}.
\end{equation}
Using (\ref{Au9a}), we get
\begin{equation}\label{Au13bb}
\frac{R}{2}\leq \frac{d}{2}+\frac{1}{d_+-d}\int_0^{d}\Big(U_Y^2-\frac{1}{d^2}\Big)dY+\frac{1}{dd_+}+ \Big(1+\frac{d}{d_+-d}\Big)\omega_0.
\end{equation}

For large $R$ the equation
$$
{\mathcal R}(s)=R
$$
has two solutions $s_+<s_c<s_-$  and they are
\begin{equation}\label{Au13ba}
s_-=\sqrt{2R}+O\Big(\frac{1}{R}\Big)
\end{equation}
and
\begin{equation}\label{Au15a}
d(s_+)=R+O(1),
\end{equation}
where $s_+\to s_0$ as $R\to\infty$.
From (\ref{Au13ba}) and (\ref{F22a}) it follows
\begin{equation}\label{Au15aa}
d=d(s_-)=\frac{1}{s_-}+O\Big(\frac{1}{s_-^3}\Big)=\frac{1}{\sqrt{2R}}+O\Big(\frac{1}{R^{3/2}}\Big).
\end{equation}
Hence
$$
\frac{1}{d^2}=s_-^2+O(1).
$$
The relations (\ref{Au15a}) and (\ref{Au15aa}) imply
$$
dd_+=\frac{\sqrt{R}}{\sqrt{2}}+O\Big(\frac{1}{\sqrt{R}}\Big).
$$
Furthermore,
$$
U_Y^2=s_-^2-2\Omega(U)
$$
and so
$$
U_Y^2(Y)=\frac{1}{d^2}+O(1).
$$
Applying these estimates to (\ref{Au13b}), we obtain
$$
R\leq O(\omega_0)+O\Big(R^{-1/2}\Big)
$$
which implies the boundedness of $R$.

Boundedness of $F$ follows from Proposition \ref{Pr26b}.
\end{proof}




\section{Proof of Proposition \ref{Pr26}}\label{S23a}



\begin{proof} 
Let $R_0=\infty$.Assume that the options (\ref{Au29ba}) and (\ref{Au29bb}) are not valid. Then there exist constants $M>0$ and $\delta>0$ such that
\begin{equation}\label{Au18av}
|\xi'(t)(X)|\leq M\;\;\mbox{for all $t>0$}
\end{equation}
 and
\begin{equation}\label{Au18avv}
|R(t)-\xi(t)(0)|\geq \delta\;\;\mbox{for all $t\geq 1$}.
\end{equation}
These inequalities implies
$$
\Psi_Y^2(X,\xi(X))(1+\xi'(X)^2)=2(R-\xi(X)).
$$
Hence
\begin{equation}\label{Okt18aa}
\Psi_Y^2(X,\xi(X))\geq\frac{2\delta}{1+M^2}.
\end{equation}
Furthermore, by Proposition 3.1\cite{KL1} and by (\ref{Okt19a}),
\begin{equation}\label{Okt18ab}
\Psi_Y(X,0)\geq \delta_1:=\sqrt{s^2-s_0^2},
\end{equation}
where $s$ is the root of the equation ${\mathcal R}(s)=R$. Since the constant $R(t)$ is uniformly bounded we obtain that $\delta_1$ does not depend on $t$.
Let us prove the following
\begin{lemma}\label{Lokt19} There exists  a constant $c_0>0$ depending on $M$, $\delta$, $R$ such that
\begin{equation}\label{Okt18a}
\Psi_Y(X,Y)\geq c_0\;\;\mbox{for all $(X,Y)\in D_\xi $}.
\end{equation}
\end{lemma}
\begin{proof} By Proposition 3.1\cite{KL1} there exists constants $\alpha\in (0,1)$ and $C>0$ depending on $M$, $\delta$ and $R$  such that
\begin{equation}\label{Okt18b}
||\Psi||_{C^{2,\alpha}(D_\xi)}\leq C.
\end{equation}
Now the estimate (\ref{Okt18a}) follows from (\ref{Okt18aa}), (\ref{Okt18ab}) and (\ref{Okt18b}) by using the Harnack principle.

\end{proof}
 Since the estimate (\ref{Okt18a}) contradicts to (\ref{Au25a}), we conclude that one of options (\ref{Au29ba}) or (\ref{Au29bb}) must be valid.

 Now let $R_0<\infty$. Then instead of (\ref{Okt18ab}) we have only positivity of $\Psi_Y(X,0)$. Using again the estimates (\ref{Okt18aa}) and (\ref{Okt18b}) together with the Harnack principle, we get that for every $\epsilon >0$ there exists $\delta>0$ such that $\Psi_Y>\delta$ if $Y>\epsilon$. This implies the second assertion of Proposition \ref{Pr26}.







\end{proof}


\section{Spectral problems}

\subsection{Spectral problems on the interval and in the strip}\label{SJu21a}

\begin{lemma}\label{LJu20a}
Let $(U,d)$ be the stream solution with the Froude number $>1$ and let $\rho_0$ be given by (\ref{Okt25a}). Then the lowest eigenvalue of the eigenvalue problem
\begin{eqnarray}\label{Ju13a}
&& -v^{''}-\omega'(U)v=\nu v\;\;\mbox{on $(0,d)$},\nonumber\\
&&v(0)=0\;\;\mbox{and}\;\;v'(d)-\rho_0 v(d)=0
\end{eqnarray}
is positive. This eigenvalue is simple and corresponding eigenfunctions does not change sign in $(0,d]$.
\end{lemma}
\begin{proof}
Let $\nu=-\tau^2$. Then it is an eigenvalue if and only if $\sigma(\tau)=0$. Since $F>1$ then the function $\sigma$ is positive and hence there are no non-positive eigenvalues  of the spectral problem (\ref{Ju13a}). The assertion on simplicity and sign of the corresponding eigenfunction is quite standard.

\end{proof}

We denote the lowest eigenvalue of the problem (\ref{Ju13a}) by $\nu_0=\nu_0(U,d)$. We have used also the notation $\nu_0(t)$ when this eigenvalue is associated with the branch of solitary waves. Then
\begin{equation}\label{Ju20a}
\nu_0(U,d)>0
\end{equation}
and this eigenvalue continuously depends on $(U,d)$.

Let
$$
\Pi_d=\{(X,Y)\,:\,X\in\Bbb R,\;0<Y<d\}.
$$
Consider the spectral problem
\begin{eqnarray}\label{Ju10a}
&&-\Delta v-\omega'(U)v=\mu v\;\;\mbox{in  $\Pi_d$},\nonumber\\
&&\partial_Yv(X,d)-\rho_0v(X,d)=0\;\;\mbox{and}\;\;v(X,0)=0.
\end{eqnarray}
We assume that the Froude number corresponding to the uniform stream solution $(U,d)$ is greater than $1$. Consider the unbounded in $L^2(\Pi_d)$ operator ${\mathcal A}=-\Delta-\omega'(U)$ introduced in  Introduction  with the domain
$$
{\mathcal D}({\mathcal A})=\{v\in H^2(\Pi_d)\,:\,v(X,0)=0\;\;\mbox{and}\;\;\partial_Yv(X,d)-\rho_0v(X,d)=0\}.
$$
Using the Fourier transform and Lemma \ref{LJu20a}, we obtain that
$$
({\mathcal A}v,v)_{L^2(\Pi_d)}=\int_{\Pi_d}(|\nabla v|^2-\omega'(U)|v|^2)dXdY-\int_{\Bbb R}\rho_0|v|^2dX\geq \nu_0\int_{\Pi_d}|v|^2dXdY
$$
for $v\in {\mathcal D}(\Pi_d)$. Therefore the spectrum of the operator ${\mathcal A}$ is continuous and coincides with the half-line $[\nu_0,\infty)$.

\subsection{Spectrum of the Frechet derivative (proof of Proposition \ref{PrOkt24a} )}\label{S31a}

Let $(\Psi,\xi)$ be a general solitary solution of (\ref{K2a}). Then
\begin{equation}\label{Ju10c}
\xi(X)=d+O(e^{-\sqrt{\nu_0}|X|})\;\;\mbox{as $X\to \pm\infty$}
\end{equation}
and
\begin{equation}\label{Ju10ca}
\Psi(X,Y)=U(Y)+O(e^{-\sqrt{\nu_0}|X|})\;\;\mbox{as $X\to \pm\infty$},
\end{equation}
where $(U,d)$ is a uniform stream solution with the Froude number $>1$.

\begin{proof} (of Proposition \ref{PrOkt24a}) If the free surface is horizontal and  defined by the uniform stream solution $(U,d)$ with $F>1$ then the result follows from Sect. \ref{SJu21a}.

Since the domain $D_\xi\setminus \Pi_d$ has the Rellich property (see \cite{Ad}) the result follows from the preservation of continuous spectrum under a compact perturbation of the domain.

To prove the second part we assume that there are infinitely many eigenfunctions $\{\Phi_j\}_{j=1}^\infty$ corresponding to eigenvalues $\leq a$ and such that
$$
||\Phi_j||_{L^2(D_\xi)}=1,\;j=1,\ldots\;\;\mbox{and}\;\;(\Phi_j,\Phi_k)_{L^2(D_\xi)}=0\;\;\mbox{for $j\neq k$}.
$$
Since the operator ${\mathcal A}$ is bounded from below, the eigenvalues are located in a bounded interval separated from the half-line $[\nu_0,\infty)$.
Using local estimates we get
$$
||\Phi_j||_{C^{2,\alpha}(D_\xi)}\leq C_1.
$$
Using asymptotical results for solutions in domains which are asymptotically a strip, we get that
$$
|\Phi_j(X,Y)|\leq C_2e^{-b|X|}\;\;\mbox{in $D_\xi$}
$$
with a positive $b$. These two estimates imply compactness of the set $\{\Phi_j\}_{j=1}^\infty$ in $L^2(D_\xi)$, which contradicts to its definition.

\end{proof}

Differentiating (\ref{K2a}) with respect to $X$, we get
\begin{eqnarray}\label{K2ad}
&&\Delta \Psi_X+\omega'(\Psi)\Psi_X=0\;\;\mbox{in ${\mathcal D}_\xi$},\nonumber\\
&&\partial_\nu \Psi_X-\rho \Psi_X=0\;\;\mbox{on ${\mathcal S}_\xi$},\nonumber\\
&&\Psi_X=0\;\;\mbox{for $Y=0$}.
\end{eqnarray}
From (\ref{Ju10ca}) it follows that $\Psi_X\in H^1({\mathcal D}_\xi)$,
$$
|\Psi_X(X,Y)|\leq Ce^{-\sqrt{\nu_0}|X|},
$$
 and  $\Psi_X$ is an eigenfunction corresponding to the zero eigenvalue but this function is odd with respect to $X$. Nevertheless it can be used in the forthcoming analysis of even eigenfunctions of the Frechet derivative.

We introduce the form on $D_a=\{u\in H^1(D_\xi)\,:\,u(X,0)=0\}$:
\begin{equation}\label{Ju25a}
a(u,v)=\int_{D_\xi}(\nabla u\cdot\nabla\overline{v}-\omega'(\Psi)u\overline{v})dXdY-\int_{-\infty}^\infty \rho u\overline{v}dS,\;\;dS=\sqrt{1+\xi'^2}dX.
\end{equation}
The next proposition contains the proof of the first assertion of Theorem \ref{T1a}.
\begin{proposition}\label{PrJ13a}  {\rm (i)} The operator ${\mathcal A}$ has at least one negative eigenvalue.

{\rm (ii)} If $\mu_0$ the lowest eigenvalue and $\Phi_0$ is a corresponding eigenfunction, then $\Phi_0(X,Y)>0$ for $Y>0$.

\end{proposition}
\begin{proof} We choose $a\in (0,\nu_0)$. Then  all  eigenfunctions of the operator ${\mathcal A}$ corresponding to eigenvalues $\leq a$ satisfy the estimate
\begin{equation}\label{Ju21a}
\int_{D_\xi}e^{-2b|X|}(|\Phi(X,Y)|^2+|\nabla\Phi|^2)dXdY\leq C\;\;\mbox{for certain positive $b$}
\end{equation}
and the minimal eigenvalue satisfies the relation
\begin{equation}\label{Ju27a}
\mu_0=\inf \frac{a(\Phi,\Phi)}{||\Phi||_{L^2}^2},
\end{equation}
where infimum is taken over all non-zero $\Phi\in D_a$ satisfying (\ref{Ju21a}). This set is not empty since the function $|\Psi_X|$ belongs to it. This implies in particular that $\mu_0\leq 0$. If $\mu_0=0$ then $|\Psi_X|$ is an eigenfunction but this function  is not sufficiently smooth. Therefore $\mu_0<0$. This proves (i).

If the eigenfunction $\Phi_0$ change sign then both functions
$$
u_\pm=\max (0,\pm\Phi)
$$
also satisfy (\ref{Ju27a}) and hence they are also eigenfunctions. But they are vanishing on a certain open set inside $D_\xi$, which leads that they are identically zero. This contradiction proves (ii).
\end{proof}




By Proposition \ref{PrJ13a} we can numerate all eigenvalues of the operator ${\mathcal A}$ according to the increasing order
$$
\mu_0<\mu_1\leq\cdots
$$
The eigenvalue $\mu_0$ is simple and $\mu_0<0$ according to Proposition \ref{PrJ13a}. Our aim is to investigate the eigenvalue $\mu_1$.

\subsection{On simplicity of the second eigenvalue of the Frechet derivative}

Since we are dealing with even functions instead the problem (\ref{K2a}) defined for all $X$ we can consider the same problem only for $X>0$ by adding the homogeneous Neumann boundary condition $\partial_Yu=0$ for $X=0$.
Let
\begin{equation}\label{Ju25aa}
a_+(u,v)=\int_{D_+}(\nabla u\cdot\nabla\overline{v}-\omega'(\Psi)u\overline{v})dXdY-\int_{0}^\infty \rho u\overline{v}dS,
\end{equation}
where
$$
D_+=\{(X,Y)\in D\,:\,X>0\}.
$$
We denote by ${\mathcal A}_+$ the unndounded operator corresponding to the form $a_+$. Clearly, the spectra of the operators ${\mathcal A}$ and ${\mathcal A}_+$ coincide and the eigenfunctions of ${\mathcal A}_+$  are restrictions to $D_+$ of the eigenfunctions of the operator ${\mathcal A}$.
The assumption (\ref{J27ba}) implies
\begin{equation}\label{J27c}
a_+(u,u)>0\;\;\mbox{for nonzero $u\in H^1(D_+)$ satisfying $u(X,\xi(X))=u(X,0)=0$}.
\end{equation}

In what follows we shall use the function
\begin{equation}\label{Sept13b}
u_*(X,Y):=\Psi_X(X,Y),
\end{equation}
which satisfies the boundary value problem (\ref{K2ad}) and  which is odd, belongs to $C^{2,\alpha}(\overline{D})$ and $u_*>0$ in $D_+$, see \cite{W}, Proposition 2.2.

The following Green's formula
 for the form $a$ will be useful in the proof of the nest proposition:
\begin{eqnarray}\label{Okt31a}
&&-\int_{D_+}(\Delta u+\omega'(\Psi)u)\overline{v}dXdY+\int_0^{\infty}(\partial_\nu u-\rho(X)u)\overline{v}dS\nonumber\\
&&=a_+(u,v)+\int_0^{\xi(0)}u_X\overline{v}|_{X=0}dY.
\end{eqnarray}
Here the function $u$ belongs to the space $H^2_0(D)$ consisting of functions from $H^2(D)$ vanishing for $Y=0$, and $v\in H^1_0(D)$.

\begin{proposition}\label{PF10a} Let $\xi$ be not identically constant. If the eigenvalue $\mu_1$ is equal to zero then  $\mu_1$ is simple. Moreover the corresponding eigenfunction has exactly two nodal sets and the nodal line separating these nodal sets has one of its end points on the curve $S_+=\{Y=\xi(X),\;\;X>0\}$.
\end{proposition}
\begin{proof} Denote by ${\mathcal X}$ the  space of real eigenfunctions corresponding to the eigenvalue $\mu_1=0$.  Since all eigenfunctions from ${\mathcal X}$ orthogonal to $\Phi_0$, each eigenfunction must change sign inside $D$. Let us show that $\dim {\mathcal X}=1$. The proof consists of several steps.

(a) First we show that every nonzero eigenfunction $w\in {\mathcal X}$ has exactly two nodal sets. Assume that there is an eigenfunction $w\in {\mathcal X}$ which has  more then two nodal sets, say $Y_j$, $j=1,\ldots, N$, $N>2$. Introduce the functions $u_j(X,Y)=w(X,Y)$ for $(X,Y)\in Y_j$ and zero otherwise. Then $u_j\in H^1_0(D)$ and
\begin{equation}\label{Sept13a}
a_+(u_j,u_k)=0\;\;\mbox{for all $k,j=1,\ldots,N$}.
\end{equation}
We choose the constant $\alpha$ such that
$$
(u_1-\alpha u_2,\Phi_0)_{L^2(D_+)}=0
$$
and observe that $a_+(u_1-\alpha u_2,u_1-\alpha u_2)=0$. Since
$$
\mu_1=\min a_+(w,w)
$$
where $\min$ is taken over $w\in H^1_0(D_+)$ satisfying $(w,\Phi_0)_{L^2(D_+)}=0$ and $||w||_{L^2(D_+)}=1$, the minimum is attained on the function $(u_1-\alpha u_2)/||u_1-\alpha u_2||_{L^2(D_+)}$. So we conclude that $u_1-\alpha u_2$ is an eigenfunction corresponding to the eigenvalue $\mu_1=0$, which is impossible since this function is zero on $Y_j$ for $j>2$.

  (b) Let $u\in {\mathcal X}$. Assume, than the nodal line of $u$ does not touch the surface $S_+$. Then one of nodal sets (say $X_1$) contains a neighborhood of $S_+$. The second nodal set (say $X_2$) is separated from $S$. Consider the function $w(X,Y)=u(X,Y)$ for $(X,Y)\in X_2$ and zero otherwise, Then $a(w,w)=0$ but by (\ref{J27c}) this form is positive  definite. Therefore $w=0$. If both end points of $\gamma$ lie outside $S$ then introduce $w_2$ which coincides with $w$ on the nodal set separated from $S$ and vanishes otherwise. Then $a(w_2,w_2)=0$ and hence $w_2=0$ also. This proves that one of end points of the line separated nodal sets must lie on $S$.


(c) Now we are in position to prove that $\dim {\mathcal X}=1$. If $\dim {\mathcal X}>1$ then there is an eigenfunction, say $w_*$ which is zero at the point $z_1=(0,\xi(0))$, which must be one of the end points of the nodal line separating two nodal sets. By b) another end-point $z_2$ of the nodal line lies on $S_+$. Denote by $Y_1$ the nodal set attached to the part of $S$ between $z_1$ and $z_2$.
Let also $Y_2$ be the remaining nodal domain. We can assume that $w_*<0$ in $Y_1$ and $w_*>0$ in $Y_2$.

Let $u_*$ be the function introduced by (\ref{Sept13b}).  Using that both functions $w_*$ and $u_*$ satisfy the problem (\ref{J17ax}) but the first  one satisfies  $\partial_Xw_*=0$ for $X=0$ and the second is subject to $u_*=0$ for $X=0$, we have, by using (\ref{Okt31a}),
\begin{equation}\label{Sept13aa}
\int_0^{\xi(0)}\partial_Xv_*(0,Y)w_*(0,Y)dY=0.
\end{equation}

Consider the function
$$
U=w_*+\beta u_*,\;\;\mbox{where $\beta>0$}.
$$
Since $u_*$ is a positive function inside $D_+$, for small $\beta$ the function $U$ has also two nodal sets $\widetilde{Y}_1$ and $\widetilde{Y}_2$, separated by a nodal line $\widetilde{\gamma}$ with end points $z_1=(0,\eta(0))$ and $\widetilde{z}_2\in S_+$ which is close to $z_2$. We assume that the nodal set $\widetilde{Y}_1$ is attached to the part of $S$ between the points $z_1$ and $\widetilde{z}_2$. Introduce the functions
\begin{eqnarray*}
&&U_1(X,Y)=U(X,Y)\;\;\mbox{if $(X,Y)\in \widetilde{Y}_1$ and zero otherwise},\\
&&U_2(X,Y)=U(X,Y)\;\;\mbox{if $(X,Y)\in \widetilde{Y}_2$ and zero otherwise}.
\end{eqnarray*}
 We choose the constant $\theta$ such that
 \begin{equation}\label{Sept13ba}
 (U_1+\theta U_2,\Phi_0)_{L^2(D)}=0.
 \end{equation}
 It is clear that $\theta\neq 0$ because of
$$
a_+(U_1,\Phi_0)=\mu_0(U_1,\Phi_0)_{L^2(D)}\neq 0.
$$
Here we have used that $\Phi_0>0$ inside $D_+$ by Proposition \ref{PrJ13a}. Furthermore due to (\ref{Sept13aa}) and the fact that $w_*$ is an eigenfunction corresponding to $\mu_1=0$, we have
$$
a_+(U_1+\theta U_2,U_1+\theta U_2)=0.
$$
This together with (\ref{Sept13ba}) implies that $U_1+\theta U_2$ is an eigenfunction corresponding to the eigenvalues $\mu_1=0$ of the spectral problem (\ref{J17ax}). But this contradicts to the smoothness properties of the function $U_1+\theta U_2$. (Compare with the argument used in a) and b)).
\end{proof}

\section{Positivity of $\mu_1(t)$ for small $t$}

We take $R$ sufficiently close to $R_c$, what is equivalent to $F$ is close to $1$. According to Sect. \ref{SJ6a}  the equation ${\mathcal R}(s)=R$ has two roots $s_+<s_c<s_-$. The root $s_-$ corresponds to the Froude number $F_-=F(s_-)>1$ and the root $s_+$ has the Froude number $F_+=F(s_+)<1$. In the latest case the dispersion equation (\ref{Okt6bb}) has a unique solution $\tau=\tau_*$. Here the Froude numbers $F_\pm$ are defined by (\ref{Okt25ba}). Furthermore, there are two uniform stream solutions corresponding to this value of $R$: $(U_-,d_-)$ and $(U_+,d_+)$, where $d_-=d(s_-)$ and $d_+=d(s_+)$.


In order to investigate the sign of $\mu_1(t)$ for small $t$ we proceed as follows. 

So we consider the branch of Stokes waves starting from the uniform stream solution $(U(s_+),d(s_+))$ and having the same Bernoulli constant $R$, by considering the period $\Lambda$ as a parameter and
$$
\Lambda_0:=\Lambda(0)=\frac{2\pi}{\tau_*},
$$
where $\tau_*$ is the root of the equation  (\ref{Okt6bb}).
 According to \cite{KL1} and \cite{KL3} there exists a branch of Stokes waves
\begin{equation}\label{Okt25bba}
(\psi(\tilde{t}),\xi(\tilde{t}),\Lambda(\tilde{t}))\in C^{k,\alpha}_{0, e,\Lambda}({\mathcal D})\times C^{k,\alpha}_{e,\Lambda}({\Bbb R}),\;\;\tilde{t}\in[0,\infty).
\end{equation}
Here $C^{k,\alpha}_{0, e}({\mathcal D})$ is the subspace of even functions vanishing for $Y=0$ in $C^{k,\alpha}({\mathcal D})$ and $C^{k,\alpha}_{0, e,\Lambda}({\mathcal D})$ is a subspace of $C^{k,\alpha}_{0, e}({\mathcal D})$ of functions of period $\Lambda$. Similarly we can define the subspace $C^{k,\alpha}_{e,\Lambda}({\Bbb R})$.

We will consider solutions to (\ref{K2a}) satisfying
\begin{equation}\label{Okt25bb}
|\xi'(X)|\leq M\;\;\mbox{for all $x\in\Bbb R$}.
\end{equation}
 We formulate a consequence of Theorem 1 and 2 from \cite{KLN17}.

\begin{theorem}\label{Tokt26} For every positive $M$ there exists $R_M\in (R_c,R_0)$ (depending also on $\omega$)
such that the following assertions are true.

(I) For every $R\in (R_0,R_M)$ problem (\ref{K2a})
has one and only one non-trivial solitary-wave solution $(\widehat{\Psi},\widehat{\xi})$.

(II) The
following inequalities hold,
$$
d_-<\xi(0;\tilde{t})\leq \widehat{\xi}(0).
$$
Moreover the mapping
$$
[0,\infty)\ni \tilde{t}\to \xi(0;\tilde{t})\in [d_-,\widehat{\xi}(0))
$$
is isomorphism and
$$
\xi(0;\tilde{t})\to \widehat{\xi}(0)\;\;\mbox{and}\;\;\Lambda(t)\to\infty\;\;\mbox{as $\tilde{t}\to\infty$}.
$$

(III) There are no other solutions to (\ref{K2a}) satisfying (\ref{Okt25bb}) exept solutions described by (\ref{Okt25bba}) and above theorem.

\end{theorem}

 Denote by $\widetilde{\mathcal A}(\tilde{t})$ the Frechet derivative of the Stokes branch (\ref{Okt25bba}) at the point $\tilde{t}$. As it was shown in Proposition 2.1\cite{Koz1a}, the first eigenvalue is always negative and the second one we denote by $\tilde{\mu}_1(\tilde{t})$. We have $\tilde{\mu}_1(0)=0$. Our first goal is to show that $\tilde{\mu}_1(\tilde{t})>0$ for small positive $\tilde{t}$.

 We will use the partial hodograph transform  for this purpose.

\subsection{Partial hodograph transform}\label{SJ29a}

In what follows we will study  branches of Stokes waves $(\Psi(X,Y;t),\xi(X;t))$ of period $\Lambda(t)$, $t\geq 0$, started from the uniform stream solution at $t=0$. The existence of such branches is established in \cite{CSst} with fixed period but variable $R$ and in \cite{KL1} for variable $\Lambda$ and fixed $R$. In our case of variable $\Lambda$
it is convenient to make the following change of variables
\begin{equation}\label{D11a}
x=\lambda X,\;\;y=Y,\;\;\lambda=\frac{\Lambda_0}{\Lambda(\tilde{t})}
\end{equation}
in order to deal with the problem with a fixed period.
As the result we get
\begin{eqnarray}\label{Okt6aa}
&&\Big(\lambda^2\partial_x^2+\partial_y^2\Big)\psi+\omega(\psi)=0\;\;\mbox{in $D_\eta$},\nonumber\\
&&\frac{1}{2}\Big(\lambda^2\psi_x^2+\psi_y^2\Big)+\eta=R\;\;\mbox{on $B_\eta$},\nonumber\\
&&\psi=1\;\;\mbox{on $B_\eta$},\nonumber\\
&&\psi=0\;\;\mbox{for $y=0$},
\end{eqnarray}
where
$$
\psi(x,y;t)=\Psi(\lambda^{-1}x,y;t)\;\;\mbox{and}\;\;\eta(x;t)=\xi(\lambda^{-1} x;t).
$$
Here all functions have the  same period $\Lambda_0$, $D_\eta$ and $B_\eta$  are the domain and the free surface  after the change of variables (\ref{D11a}).

We assume that
$$
\psi_y>0\;\;\mbox{in $\overline{D_\eta}$}
$$
and use the variables
$$
q=x,\;\;p=\psi.
$$
Then
$$
q_x=1,\;\;q_y=0,\;\;p_x=\psi_x,\;\;p_y=\psi_y,
$$
and
\begin{equation}\label{F28b}
\psi_x=-\frac{h_q}{h_p},\;\;\psi_y=\frac{1}{h_p},\;\;dxdy=h_pdqdp.
\end{equation}

System (\ref{Okt6aa}) in the new variables takes the form
\begin{eqnarray}\label{J4a}
&&\Big(\frac{1+\lambda^2h_q^2}{2h_p^2}+\Omega(p)\Big)_p-\lambda^2\Big(\frac{h_q}{h_p}\Big)_q=0\;\;\mbox{in $Q$},\nonumber\\
&&\frac{1+\lambda^2h_q^2}{2h_p^2}+h=R\;\;\mbox{for $p=1$},\nonumber\\
&&h=0\;\;\mbox{for $p=0$}.
\end{eqnarray}
Here
$$
Q=\{(q,p)\,:\,q\in\Bbb R\,,\;\;p\in (0,1)\}.
$$
The uniform stream solution corresponding to the solution $U$ of (\ref{X1}) is
\begin{equation}\label{M4c}
H(p)=\int_0^p\frac{d\tau}{\sqrt{s^2-2\Omega(\tau)}},\;\;s=U'(0)=H_p^{-1}(0).
\end{equation}
One can check that
\begin{equation}\label{J18aaa}
H_{pp}-H_p^3\omega(p)=0
\end{equation}
or equivalently
\begin{equation}\label{J18aa}
\Big(\frac{1}{2H_p^2}\Big)_p+\omega(p)=0.
\end{equation}
Moreover it satisfies the boundary conditions
\begin{equation}\label{M4ca}
\frac{1}{2H_p^2(1)}+H(1)=R,\;\;H(0)=0.
\end{equation}
The Froude number in new variables can be written as
$$
\frac{1}{F^2}=\int_0^1H_p^3dp.
$$

 Then according to Theorem 2.1, \cite{KL1} there exists a branch of solutions to (\ref{J4a})
\begin{equation}\label{J4ac}
h=h(q,p;t):[0,\infty)\rightarrow C^{2,\gamma}_{pe}(\overline{Q}),\;\;\lambda=\lambda(t):[0,\infty)\rightarrow (0,\infty),
\end{equation}
which  has a real analytic reparametrization locally around each $t\geq 0$.

\subsection{Bifurcation equation}

In order to find bifurcation points and bifuracating solutions we put $h+w$ instead of $h$ in (\ref{J4a}) and introduce the operators
\begin{eqnarray*}
&&{\mathcal F}(w;\tilde{t})=\Big(\frac{1+\lambda^2(h_q+w_q)^2}{2(h_p+w_p)^2}\Big)_p
-\Big(\frac{1+\lambda^2h_q^2}{2h_p^2}\Big)_p\\
&&-\lambda^2\Big(\frac{h_q+w_q}{h_p+w_p}\Big)_q+\lambda^2\Big(\frac{h_q}{h_p}\Big)_q
\end{eqnarray*}
and
$$
{\mathcal G}(w;\tilde{t})=\frac{1+\lambda^2(h_q+w_q)^2}{2(h_p+w_p)^2}-\frac{1+\lambda^2h_q^2}{2h_p^2}+w
$$
acting on $\Lambda_0$-periodic, even functions $w$ defined in $Q$. After some cancelations we get
$$
{\mathcal F}={\mathcal J}_p+{\mathcal I}_q,\;\;{\mathcal G}={\mathcal J}+w,
$$
where
$$
{\mathcal J}=\frac{\lambda^2h_p^2(2h_q+w_q)w_q-(2h_p+w_p)(1+\lambda^2h_q^2)w_p}{2h_p^2(h_p+w_p)^2}
$$
and
$$
{\mathcal I}=-\lambda^2\frac{h_pw_q-h_qw_p}{h_p(h_p+w_p)}.
$$
Both these functions are well defined  for small $w_p$.
Then the problem for finding solutions close to $h$ is the following
\begin{eqnarray}\label{F19a}
&&{\mathcal F}(w;\tilde{t})=0\;\;\mbox{in $Q$}\nonumber\\
&&{\mathcal G}(w;\tilde{t})=0\;\;\mbox{for $p=1$}\nonumber\\
&&w=0\;\;\mbox{for $p=0$}.
\end{eqnarray}

Furthermore, the Frechet derivative (the linear approximation of the functions ${\mathcal F}$ and ${\mathcal G}$) is the following
\begin{equation}\label{J4aa}
Aw=A(\tilde{t})w=\Big(\frac{\lambda^2h_qw_q}{h_p^2}-\frac{(1+\lambda^2h_q^2)w_p}{h_p^3}\Big)_p-\lambda^2\Big(\frac{w_q}{h_p}-\frac{h_qw_p}{h_p^2}\Big)_q
\end{equation}
and
\begin{equation}\label{J4aba}
{\mathcal N}w={\mathcal N}(\tilde{t})w=(N w-w)|_{p=1},
\end{equation}
where
\begin{equation}\label{J4ab}
N w=N(\tilde{t})w=\Big(-\frac{\lambda^2h_qw_q}{h_p^2}+\frac{(1+\lambda^2h_q^2)w_p}{h_p^3}\Big)\Big|_{p=1}.
\end{equation}
The eigenvalue problem for the Frechet derivative, which is important for the analysis of bifurcations of the problem
(\ref{F19a}), is the following
\begin{eqnarray}\label{M1a}
&&A(\tilde{t})w=\mu w\;\;\mbox{in $Q$},\nonumber\\
&&{\mathcal N}(\tilde{t})w=0\;\;\mbox{for $p=1$},\nonumber\\
&&w=0\;\;\mbox{for $p=0$}.
\end{eqnarray}

For $\tilde{t}=0$ and $\mu=0$ this problem becomes
\begin{eqnarray}\label{J15a}
&&A_0w:=-\Big(\frac{w_{p}}{H_p^3}\Big)_p-\Big(\frac{w_{q}}{H_p}\Big)_q=0\;\;\mbox{in $Q$},\nonumber\\
&&B_0w:=-\frac{w_{p}}{H_p^3}+w=0\;\;\mbox{for $p=1$},\nonumber\\
&&w=0\;\;\mbox{for $p=0$}.
\end{eqnarray}
Since the function $H$ depends only on $p$ this problem admits the separation of variables and its solutions are among the functions
\begin{equation}\label{J15aa}
v(q,p)=\alpha(p)\cos (\tau q),\;\;\mbox{where}\;\;\tau=k\tau_*,\;\;k=0,1,\ldots.
\end{equation}
According to \cite{Koz1} the function (\ref{J15aa}) solves (\ref{J15a}) if and only if
$$
\alpha(p)=\gamma(H(p);\tau)H_p,
$$
where the function $\gamma(Y;\tau)$ solves the euation (\ref{Okt6b}) and $\sigma(\tau)=0$. Therefore if $\tau\neq \tau_*$ then the problem (\ref{J15a}) has no non-trivial solutions. If $\tau=\tau_*$ then
the kernel of the above operator is one dimensional in the class of $\Lambda_0$ periodic, even function and it is given by
\begin{equation*}
v=\alpha(p)\cos(\tau_*q),  \alpha(p)=\gamma(H(p);\tau_*)H_p.
\end{equation*}

We will need also the problem
\begin{eqnarray}\label{M6a}
&& -\Big(\frac{u_{p}}{H_p^3}\Big)_p+\frac{\tau^2u}{H_p}=F\;\;\mbox{on (0,d)}\nonumber\\
&&u(0)=0,\;\;-\frac{u_{p}}{H_p^3}+u=c\;\;\mbox{for $p=1$},
\end{eqnarray}
where $F\in C^{0,\alpha}([0,1])$ and $c$ is a constant.
Clearly this problem is
elliptic and uniquely solvable for all $\tau\geq 0$, $\tau\neq\tau_*$, the problem (\ref{M6a}) has a unique solution in  $C^{2,\alpha}([0,1])$. This solution is given by
$$
u(p)=v(H(p))H_p(p),
$$
where $v(Y)$ solves the problem (\ref{J7b}) with $f=F(H(y))$ and $g=c$.

The functions $\tilde\mu_1(\tilde{t})$ and $\lambda(\tilde{t})$ have the representations near $\tilde{t}=0$:
$$
\tilde{\mu}_1(\tilde{t})=\mu_2\tilde{t}^2+O(\tilde{t}^3)
$$
and
\begin{equation}\label{M5a}
\lambda(\tilde{t})=1+\lambda_2\tilde{t}^2+O(\tilde{t}^3).
\end{equation}
It was proved in \cite{Koz23} that
\begin{equation}\label{J3aa}
-4\lambda_2\tau_*^2\int_0^d\gamma(Y;\tau_*)^2dY=\mu_2\int_0^d \gamma(Y;\tau_*)^2\frac{dY}{\Psi_Y},
\end{equation}
This implies that the sign of $\mu_2$ is opposite to the sign of $\lambda_2$. Our main result of this section is the following
\begin{proposition}\label{PrNov15} There exists $R_1>R_c$ such that
$$
\mu_2>0\;\;\mbox{for $R\in (R_c,R_1)$}.
$$
\end{proposition}

Due to (\ref{J3aa}) it is sufficient to prove that $\lambda_2<0$ for $R\in (R_c,R_1)$. We note that both quantities $\lambda_2$ and $\mu_2$ depemds on $\tau_*$, which depends on $R$ and $\omega$. If $R$ is close to $R_c$ then $\tau_*$ is close to zero.

\subsection{Stokes waves for small $\tilde{t}$}\label{SNov11a}

Here we consider asymptotics of solutions of (\ref{J4ac}) for small $t$. For this purpose we take
$$
h=H(p)
$$
and represent the solutions in the form
$$
H(p)+w(q,p,\tilde{t}),\;\;w=\tilde{t}v,
$$
where 
\begin{equation}\label{Okt11b}
v(q,p;\tilde{t})=v_0(q,p)+\tilde{t}v_1(q,p)+\tilde{t}^2v_2(q,p)+\cdots
\end{equation}

In this case
$$
{\mathcal J}=A_1\Big(1+\frac{w_p}{H_p}\Big)^{-2}+A_2\Big(1+\frac{w_p}{H_p}\Big)^{-2},
$$
where
$$
A_1=-\frac{w_p}{H_p^3}
$$
and
$$
A_2=\frac{\lambda^2w_q^2}{2H_p^2}-\frac{w_p^2}{2H_p^4}.
$$
Therefore
$$
{\mathcal J}={\mathcal J}_1+{\mathcal J}_2+{\mathcal J}_3+O(t^4),
$$
where
$$
{\mathcal J}_1=A_1,
$$
$$
{\mathcal J}_2=A_2-2\frac{w_p}{H_p}A_1=\frac{\lambda^2w_q^2}{2H_p^2}+\frac{3}{2}\frac{w_p^2}{H_p^4}
$$
and
$$
{\mathcal J}_3=3\frac{w_p^2}{H_p^2}A_1-2\frac{w_p}{H_p}A_2=-2\frac{w_p^3}{H_p^5}-\frac{w_pw_q^2}{H_p^3}.
$$
Furthermore
$$
{\mathcal I}=-\lambda^2\frac{w_q}{H_p}\Big(1+\frac{w_p}{H_p}\Big)^{-1}
={\mathcal I}_1+{\mathcal I}_2+{\mathcal I}_3+O(t^4).
$$
Here
$$
{\mathcal I}_1=-\lambda^2\frac{w_q}{H_p},\;\;{\mathcal I}_2=\lambda^2\frac{w_qw_p}{H^2_p},\;\;{\mathcal I}_3(w)=-\lambda^2\frac{w_qw_p^2}{H^3_p}.
$$

In what follows we shall use the notation
$$
\tau_*=\frac{2\pi}{\Lambda_*},\;\;\Lambda_*=\Lambda(0).
$$

Inserting (\ref{Okt11b}) and (\ref{M5a}) into (\ref{F19a}) and
equating terms of the same power with respect to $t$, we get
\begin{eqnarray*}
&&A_0v_0:=-\Big(\frac{v_{0p}}{H_p^3}\Big)_p-\Big(\frac{v_{0q}}{H_p}\Big)_q=0\;\;\mbox{in $Q$},\\
&&B_0v_0:=-\frac{v_{0p}}{H_p^3}+v_0=0\;\;\mbox{for $p=1$},\\
&&v_0=0\;\;\mbox{for $p=0$}.
\end{eqnarray*}
As we have shown in previous section the kernel of the above operator is one dimensional  and is generated by the function
\begin{equation}\label{Okt12a}
v_0=\alpha_0(p)\cos(\tau_*q),\;\;\alpha_0=\gamma(H(p);\tau_*)H_p.
\end{equation}

The next term in the asymptotics satisfies the boundary value problem
\begin{eqnarray}\label{Okt13a}
&&A_0v_1+\Big(\frac{v_{0q}^2}{2H_p^2}+\frac{3}{2}\frac{v_{0p}^2}{H_p^4}\Big)_p+\Big(\frac{v_{0q}v_{0p}}{H^2_p}\Big)_q=0\;\;\mbox{in $Q$},\nonumber\\
&&B_0v_{1}+\frac{v_{0q}^2}{2H_p^2}+\frac{3}{2}\frac{v_{0p}^2}{H_p^4}=0 \;\;\mbox{for $p=1$},\nonumber\\
&&v_1=0\;\;\mbox{for $p=0$}.
\end{eqnarray}
The solution of this problem, orthogonal to $v_0$ in $L^2$, is given by
\begin{equation}\label{Okt12aa}
v_1=\alpha_1(p)+\beta_1(p)\cos(2\tau_* q),
\end{equation}
where $\alpha_1$ and $\beta_1$ satisfy the problem (\ref{M6a}) with $\tau=0$ and $\tau=2\tau_*$ respectively with certain right-hand sides.
 Further, the term $v_2$ is fond from the following problem
\begin{eqnarray*}
&&A_0v_2+\Big(\frac{v_{0q}v_{1q}}{H_p^2}+\frac{3v_{0p}v_{1p}}{H_p^4}+{\mathcal J}_3(v_0)\Big)_p
+\Big(\frac{v_{1q}v_{0p}+v_{0q}v_{1p}}{H^2_p}+{\mathcal I}_3(v_0)\Big)_q=2\lambda_2\Big(\frac{v_{0q}}{H_p}\Big)_q\;\;\mbox{on $Q$},\\
&&B_0v_2+\frac{v_{0q}v_{1q}}{H_p^2}+\frac{3v_{0p}v_{1p}}{H_p^4}+{\mathcal J}_3(v_0)=0\;\;\mbox{for $p=1$}\\
&&v_2(q,0)=0.
\end{eqnarray*}

The solvability condition for the last problem has the form
\begin{eqnarray}\label{Okt4a}
&&2\lambda_2\int_\Omega\frac{v_{0q}^2}{H_p}dqdp-\int_\Omega \Big(\Big(\frac{v_{0q}v_{1q}}{H_p^2}+\frac{3v_{0p}v_{1p}}{H_p^4}\Big)v_{0p}+\frac{v_{0q}v_{1p}+v_{1q}v_{0p}}{H^2_p}v_{0q}\Big)dqdp\nonumber\\
&&+\int_\Omega\Big(\Big(\frac{2v_{0p}^3}{H_p^5}+\frac{v_{0p}v_{0q}^2}{H_p^3}\Big)v_{0p}+\frac{v_{0p}^2v_{0q}}{H_p^3}v_{0q}\Big)dqdp=0.
\end{eqnarray}
This relation can be used to find $\lambda_2$ for small $\tau_*$.  The function $v_2$ has the form
\begin{equation}\label{J12a}
v_2=\alpha_2(p)\cos(\tau_* q)+\beta_2(p)\cos(3\tau_* q),
\end{equation}
where $\alpha_2$ and $\beta_2$ satisfy the problem (\ref{M6a}) with $\tau=\tau_*$ and $\tau=3\tau_*$ respectively with certain right-hand sides.

Thus we have shown that $\lambda$ and $v$ have the form (\ref{M5a}) and (\ref{Okt11b}) respectively. More exactly $v_0$ is given by (\ref{Okt12a}), $v_1$ is represented as (\ref{Okt12aa}) and $v_2$ by (\ref{J12a}).

\subsection{Coefficient $\lambda_2$ for small $\tau_*$ and proof of Proposition \ref{PrNov15}}\label{SJu5}

Here we assume that $0<\tilde{t}\ll \tau_*\ll 1$. The relation (\ref{Okt13a}) for finding $v_1$ has the form
\begin{eqnarray}\label{Okt13aa}
&&A_0v_1+\frac{3}{2}\Big(\frac{v_{0p}^2}{H_p^4}\Big)_p=O(\tau_*^2)\;\;\mbox{in $Q$},\nonumber\\
&&B_0v_{1}+\frac{3}{2}\frac{v_{0p}^2}{H_p^4}=O(\tau_*^2)\;\;\mbox{for $p=1$},\nonumber\\
&&v_1=0\;\;\mbox{for $p=0$}.
\end{eqnarray}
We are looking for the solution in the form
$$
v_1(q,p)=a_1(p)\cos^2(\tau_*q).
$$
 Then $a_1$ satisfies the equation
$$
a_{1p}=\frac{3}{2}\frac{\alpha_{0p}^2}{H_p}+H_p^3c_1,
$$
where $c_1$ is a constant. Integrating this relation from $0$ to $p$ we get
$$
a_{1}(p)=\int_0^p\Big(\frac{3}{2}\frac{\alpha_{0p}^2}{H_s}+H_s^3c_1\Big)ds.
$$
From the boundry condition for $p=1$ we get
$$
a_1(1)-c_1=0.
$$
Therefore
$$
c_1\Big(1-\int_0^1H_s^3ds\Big)=\int_0^1\frac{3}{2}\frac{\alpha_{0p}^2}{H_s}ds.
$$
Since $F<1$ and the left hand side is equal to $c_1(1-F^{-2})$ the coefficient $c_1$ is negative and
\begin{equation}\label{Au19c}
c_1=(1-F^{-2})^{-1}\int_0^1\frac{3}{2}\frac{\alpha_{0p}^2}{H_s}ds.
\end{equation}

Integrating relation (\ref{Okt4a}) with respect to $q$ and keeping terms containing $c_1$ we get
\begin{equation}\label{Nov16a}
\lambda_2\tau_*^2\int_0^1\frac{\alpha_{0}^2}{H_p}dp=c_1\frac{9}{8}\Big(\int_0^1\Big(\frac{\alpha_{0p}^2}{H_p}+O(\tau_*)\Big)dp\Big)+O(1).
\end{equation}
Since $F$ is less than $1$ and close to $1$ the number $c_1$ is negative and large in the absolute value when $R_1$ is close to $R_c$. Therefore $\lambda_2<0$, which proves Proposition \ref{PrNov15}.




\section{Proofs of Theorem \ref{T1a} and Proposition \ref{PrNov11}}

\subsection{The main steps of the proof of Theorem \ref{T1a}}.

The assertions (i) and (iii) in Theorem \ref{T1a} follow from Propositions \ref{Ju25a} nd \ref{PF10a} respectively. Therefore it remains to prove the assertion (ii).


The proof of the assertion (ii) consists of three steps.

(i) In the paper \cite{W}, Sect.4, it was proved that the Frechet derivative at the elements of the solitary branch (\ref{J3b}) is invertible for small $t$. This means that $\mu_1(t)\neq 0$ for small $t$.

(ii) Let us take a small $t$ and consider the solitary wave $(\Psi,\xi,R)=(\Psi(t),\xi(t),R(t))$. Then the Bernoulli constant $R$ is close to $R_c$ and we can construct the branch of Stokes waves (\ref{Okt25bba}) starting from the uniform stream solution $(U_+,d_+)$. We prove that the  eigenvalue $\tilde{\mu}_1(\tilde{t})$ of the Frechet derivative is positive for all values of $\tilde{t}$ on this branch of Stokes waves.

(iii) By Theorem \ref{Tokt26} and by (i), (ii) the firs eigenvalue is positive also for the Frechet derivative at the solitary solution $(\Psi,\xi,R)$.

Among these steps only the second step must be proved.

\subsection{Solitary waves: the eigenvalue $\mu_1(t)$ for small $t$, proof of (ii)}

Consider the Frechet derivative on the Stokes branch (\ref{Okt25bba}). As it was proved in \cite{Koz1a} the lowest eigenvalue $\tilde{\mu}_0(\tilde{t})$ of the Frechet derivative is always negative and according to \cite{Koz1a} the eigenvalue $\tilde{\mu}_1(\tilde{t})$ cannot change sign. Indeed if it changes  then we can consider the first bifurcation point and according to \cite{Koz1a} there are infinitely many subharmonic bifurcations near such point, which contradicts to the uniqueness assertion in Theorem \ref{T1a}. Thus $\tilde{\mu}_1(\tilde{t})\geq 0$ for all $\tilde{t}\geq 0$. Assume that there exists a sequence  $\tilde{t}_j\to\infty$ as $j\to\infty$  such that $\tilde{\mu}_1(\tilde{t_j})=0$ for all $j$. This implies that the second eigenvalue $\mu_1(t)$ is zero, but this contradicts to (i). Hence the eigenvalue $\tilde{\mu}_1(\tilde{t})> 0$ for large $t$ and the same is true for the second eigenvalue $\mu_1$. This proves (ii) and consequently Theorem \ref{T1a}.



\subsection{Proof of Proposition  \ref{PrNov11}}

According to the assumption of  Proposition  \ref{PrNov11} there exists a sequence $\{t_j\}_{j01}^\infty$ such that $R(t_j)\to R_*$ and
$$
\xi_j(0)\to R\;\;\mbox{as $j\to\infty$}.
$$
 In this case the limit point $(0,\xi_*(0),R)_*$ is a stagnation point and  the limit configuration $(\Psi_*,\eta_*,F)$ satisfies: there is the angle of $120^\circ$ at the stagnation point.

Reasoning as in Sect. 1.6 \cite{Koz1}, we get
\begin{equation}\label{D21a}
|\nabla\Psi_j(X,Y)|\leq C(R_*,\omega_0),\;\;\mbox{for $(X,Y)\in \overline{{\mathcal D}_{\xi_j}}$},
\end{equation}
where $C$ depends only on $R_*$ and $\omega_0$, $\omega_0=\max_{0\leq p\leq 1}\omega(p)$. One can choose a subsequence of $\{\xi_j\}$ which is convergent in $C^{0,\alpha}_{pe}(\Bbb R)$ for any $\alpha\in (0,1)$, to a function $\xi_*\in C^{0,1}_{pe}(\Bbb R)$. By  (\ref{D21a}) we can assume also that the sequence $\{\widetilde{\Psi}_j\}$, where $\widetilde{\Psi}_j$ is the extension of $\Psi_j$ by $1$ for $Y>\xi(X)$ and by $0$ for $Y<0$,  is also convergent in $L^\infty (Q_{R_*,a})$ to a function $\widetilde{\Psi}_*$, where
$$
Q_{R,a}=\{(X,Y)\,:\,X\in (-a,a),\;Y\in (0,R)\}.
$$
Moreover
$$
\nabla\widetilde{\Psi}_j\; \mbox{weak}^*\;\nabla\widetilde{Psi_*}\;\mbox{in}\; L^\infty (Q_{R,a,b}).
$$
Then the limit functions $(\Psi_*,\xi_*,R_*)$ satisfies (\ref{K2a}) in a weak sense, see \cite{Varv}. One can verify that all conditions of Theorem 5.2, Varvaruca \cite{Varv}, are satisfied and
according to that theorem  there are two options for the limit function $\xi_*$:
\begin{equation}\label{F24a}
\lim_{X\to 0+}\frac{\xi_*(X)}{X}=-\frac{1}{\sqrt{3}}\;\;\mbox{or}\;\;\lim_{X\to 0}\frac{\xi_*(X)}{X}=0.
\end{equation}
 It is assumed in the theorem that  the first option holds, i.e. the limit Stokes wave has an opening angle  $120$ degree at the stagnation points.

 \begin{remark}\label{PrN26} It is proved in \cite{KL5} that the second option in (\ref{F24a}) is imposible. So the first option in (\ref{F24a}) is always satisfied when we approach a stagnation point. Therefore one can remove the text "with angle $120^\circ$ at the crest"
 in Proposition \ref{PrNov11}.
 \end{remark}

In the paper \cite{Koz1} it was shown that the number of negative eigenvalues of the Frechet derivative tends to infinity if we consider the eigenvalue problem in a bounded domain between $Y=\pm a$ with the homogeneous Neuman boundary condition on the new boundary corresponding to the boundary. Since the corresponding eigenfunctions decay outside a small neighborhood of the crest, we can use these function to show that the same is true for the eigenvalue problem in $D(t_j)$, i.e. the number of eigenvalues for the Frechet derivative in  $D(t_j)$ tends to $\infty$ as $j\to\infty$.

Denote by ${\mathcal A}_j$ the Frechet derivative at the point $t_j$. As a result of the above arguments we get

\begin{proposition}\label{PJu30}
 Let $N_j$ be the number of negative eigenvalues of the operator ${\mathcal A}_j$. Then $N_j\to\infty$ as $j\to\infty$.

\end{proposition}

Now using Theorem \ref{T1a} and the above proposition we obtain that the eigenvalue $\mu_1(t)$ is non-negative on $(0,t_*)$ for certain positive $t_*$ and then negative for $t\in (t_*,t_*+\epsilon)$ for a certain positive $\epsilon$. Since $\mu_1(t_*)=0$ is a simple eigenvalue and $\mu_1$ is analytic in a neighborhood of $t_*$ we get that $\mu_1$ is positive for $t<t_*$ and $t$ close to $t_*$. Therefore the crossing number at $t_*$ is $1$. This proves Proposition \ref{PrNov11}.



\section{Acknowledgments}

This material is based upon work supported by the Swedish Research Council under grant no. 2021-06594 while the author was in residence at Institut Mittag-Leffler in Djursholm, Sweden during Oct 08- Oct 29, 2023.

\section{Data availability}

No data was used for the research described in the article.



\section{References}

{

\end{document}